\documentclass[12pt]{article}

\usepackage{a4wide}

\usepackage{mathptmx}          
\usepackage{helvet}            
\usepackage{courier}           
\usepackage{makeidx}           
\usepackage{graphicx}          
\usepackage{multicol}          
\usepackage[bottom]{footmisc}  

\usepackage{hyperref}
\hypersetup{colorlinks=true, urlcolor=blue, citecolor=cyan}
\usepackage[comma,authoryear]{natbib}
\usepackage[english]{babel}
\usepackage{amsmath}
\usepackage{amssymb}
\usepackage{dsfont}
\usepackage{latexsym}
\usepackage{amssymb}
\usepackage{verbatim}

\makeindex                    

\newtheorem{proposition}{Proposition}
\newtheorem{remark}{Remark}
\newtheorem{example}{Example}
\newtheorem{theorem}{Theorem}
\newtheorem{lemma}{Lemma}

\newtheorem{proof}{Proof}

\newcommand{\bfSigma}{\Sigma}

\begin{document}

\title{Sampling Plans for Control-Inspection Schemes Under Independent and Dependent Sampling Designs With Applications to Photovoltaics}

\author{Ansgar Steland \\ \ \\ 
Institute of Statistics \\ RWTH Aachen University, Germany\\
{\normalsize steland@stochastik.rwth-aachen.de}}

\maketitle

\abstract{The evaluation of produced items at the time of delivery is, in practice,
usually amended by at least one inspection at later time points. We extend the methodology of acceptance sampling for variables for arbitrary unknown distributions when additional sampling information is available to such settings. 
Based on appropriate approximations of the operating characteristic, we derive new acceptance sampling plans that control the overall operating characteristic. The results cover the case of independent sampling as well as the case of dependent sampling. In particular, we study a modified panel sampling design and the case of spatial batch sampling. The latter is advisable in photovoltaic field monitoring studies, since it allows to detect and analyze local clusters of degraded or damaged modules. Some finite sample properties are examined by a simulation study, focusing on the accuracy of estimation.}


\section{Introduction}\label{t19:sec:1}

The acceptance sampling problem deals with the construction of sampling plans for  inspection, in order to decide, using a minimal number of measurements, whether a lot  (or shipment) of produced items should be accepted or rejected. Our approach is 
motivated by applications in photovoltaics, where the distribution of the relevant quality features, in particular the power output of solar panels, is typically non-normal, unknown and cannot be captured appropriately by parametric models.  However, in photovoltaics additional measurements from the production line are available and can be used to construct acceptance sampling plans. 

Let $ X $ represent a control measurement of produced item, with distribution function $ F $ with a finite fourth moment, e.g. the power output of a photovoltaic module. It is classified as non-conforming, defective or out-of-spec, if $ X \le \tau $, where $ \tau $ is a (one-sided) specification limit, usually defined as $ \tau = \mu^*(1-\varepsilon) $, where $ \mu^* $ is the target (or nominal) mean and $ \varepsilon \in (0,1) $ the tolerance. If there were no randomness, $ X = \mu $, where $\mu=E(X) $ is the true mean of $X$. Then items are non-conforming if $ \mu - \tau \le 0 $ and, clearly, we should reject the lot, if and only if $
\mu - \tau \le 0 $. But if the distribution of $X$ is not degenerated, it is reasonable to replace $ \mu $ by its unbiased canonical estimator $ \widehat{\mu} = \overline{X} $ to form a decision rule, and thus to reject the lot if $ \overline{X} - \tau \le c $, where
the critical value $ c> 0 $ accounts for the estimation error.

The  fraction of non-conforming (out-of-spec) modules (or items) corresponding to the above definition is then given by
\[
  p = P( X_1 \le \tau ) = F( \tau ).
\]
It is usually regarded as the quantity of interest in quality control, although it makes no assertion about how far away from the specification the non-conforming items are. However, it is worth mentioning that the fraction of non-conforming modules is directly related to the resulting costs, since it determines the number of modules one has to repair or replace in case of a total inspection. For these reasons, one aims at the determination of control procedures that allow to infer whether the fraction $p$, also called quality level, is acceptable, i.e. $ p < AQL $, or not, i.e. $ p > RQL $. Here 
$ 0 < AQL < RQL < 1 $ denote the \textit{acceptable qualilty limit} (AQL) and the \textit{rejectable quality limit} (RQL).

It is known that the probability of acceptance for the above rule based on $ \overline{X} - \tau $ is a function of the fraction defectives $p$, but it depends on the unknown distribution function $ F $. In photovoltaics and presumably other areas as well, we are given additional data from the production line that can be used to estimate unknowns.

In this paper, we construct sampling plans for the following situation: We assume that a control sample is drawn at the time of delivery of the modules, i.e., when the modules are new and unused, in order to ensure that shipments that are out-of-spec are identified and are not delivered to the customers. The sampling plan used at this first stage of the procedure is constructed using an additional sample from the production line, as discussed above. We further assume that shipments that passed the first-stage acceptance sampling procedure are inspected at a later time point, in order to check whether they are still in agreement with the quality requirements after some defined period of operation. At this second stage a further sample is taken, which is combined with the sample information from the first stage. This is done in order to avoid that a close decision in favor of acceptance, i.e. when the first stage control statistic has attained a value close to the critical value, results again in an acceptance due to the fact that the second stage sample size is relatively small.

At first glance, it seems that the approach is a double sampling procedure. Recall that the idea of double sampling plans is to give a questionable lot a second chance or, put differently, allow for quick acceptance of very good lots: If the test statistic (say, e.g. the number of defectives) is small, say smaller than $ a $, one accepts, if it is too large, say larger than $ b $, one rejects, and if it is between $a$ and $b$ one draws a second sample and bases the decision on the enlarged sample. Our approach is related in that we aim at re-using the sample information from the first-stage control sample, but in our approach the second sample is not taken at the same time instant, but at the inspection time, and the decision to accept or reject at the first stage is only based on the first sample. Indeed, we have in mind that there may be a substantial time lag between the two stages. 

The fact that in practice one prefers to take repeated measures at inspection time, i.e. of those items already selected, complicates the design of appropriate procedures, since now the samples at different time points are not independent. Thus, we extent the required theoretical results to the case of dependent sampling under quite general conditions. We propose a panel-based sampling scheme, where
the items selected at the first stage form the basis of the second stage sample, which is enlarged by new items if necessary. A further important sampling design is spatial batch sampling. Here the batches of observations may be correlated, for example since their spatial closeness implies that they carry the same factors that may affect quality measurements. We show that our results are general enough to apply to such a sampling design as well.

The rest of the paper is organized as follows. Section~\ref{t19:sec: Pre} discusses related work and applications.
Section~\ref{t19:sec: method} introduces the two-stage acceptance sampling framework, in particular the operating characteristic curves that define the statistical behaviour of our two-stage procedure, and discusses our model for the two-stage setting with a control sample, an inspection sample and an additional sample from the production line. In Section~\ref{t19:sec: approx}, we provide the asymptotic results that allow us to construct valid acceptance sampling plans that control the overall operating characteristic curve. Those results cover expansions of the control statistics, their joint asymptotic normality and approximations of the operating characteristics. We provide results for the case that the control sample and the inspection sample are independent as well as for the more general and realistic case that the samples are dependent. Computational issues
are discussed in Section~\ref{t19:sec: comp}. Lastly,
Section~\ref{t19:sec: sims} presents results from a simulation study.

\section{Preliminaries}\label{t19:sec: Pre}

\subsection{Related work}

The acceptance sampling problem dates back to the seminal contributions of Dodge and has been studied since then to some extent. For a general overview of classical procedures and their implementations in standards we refer to recent monograph \citet{SchillingNeubauer09}. For the Gaussian case, optimal plans have been constructed by \citet{LiebermannResnikoff1955}, see also \citet{BruhnSuhrKrumbholz1991},
 \citet{FeldmannKrumbholz2002}, where the latter paper studies double sampling plans for normal and exponential data, and the references given there.
Their lack of robustness with respect to departures from normality has been discussed in \citet{KoesslerLenz1997}. \citet{Koessler1999} used a Pareto-type tail approximation of the operating characteristic combined with maximum likelihood estimation, in order to estimate the fraction of defectives and then constructed sampling plans using the asymptotic distribution of that estimate, when the lot is accepted if the
estimated fraction of defectives is small enough. The methods works, if the tails are not too short. Since in industrial applications large production lots are usually classified in classes, the case of non-normal but compactly supported distributions 
deserves attention. For such distributions, approximations based on the asymptotic normality of sample means are a convenient and powerful tool for the construction of sampling plans, having in mind that $t$-type statistics are a natural choice to decide in favor of acceptance or rejection of a lot, as discussed above. Thus, recent works focused on $t$-type test statistics resembling the statistic used by the optimal procedure under normality.

Sampling plans for variables inspection when the underlying quality variable has an arbitrary continuous distribution with finite fourth moment and the related estimation
theory based on the sample quantile function of an additional sample has been studied in \citet{StelandZaehle2009} employing empirical process theory. For historic samples, i.e. samples having the same distribution as the control sample, a simplified proof using the Bahadur representation can be found in \citet{MeisenPepelyshevSteland2012}. In the present work, it is shown that this method of proof extends to the case of a difference in location between the additional sample and the control sample as studied in
\citet{StelandZaehle2009}. 
Further results and discussions on acceptance sampling for photovoltaic data and applications can be found in \citet{HerrmannStelandHerff10}. \citet{HerrmannSteland2010} have shown that the accuracy of such acceptance sampling plans using additional samples from the production line can be substantially improved by using smooth quantile estimators such as numerically inverted integrated  cross-validated kernel density estimators. The construction of procedures using the singular spectral analysis (SSA) approach with adaptively estimated parameters has been recently studied by \citet{GolyandinaPepelyshevSteland2012}. 
Bernstein-Durrmeyer polynomial estimators are well known as a general purpose approach to estimation that provides quite smooth estimators. The relevant theory as well as their application to the construction of acceptance sampling plans with one-sided specification limits has been investigated in \citet{PepelyshevRafajlowiczSteland2013}. 
The extension of the methodology to two-sided specification limits and numerical results focusing on photovoltaic applications can be found in \citet{AvellanPepelyshevSteland2013}.

\subsection{Applications in photovoltaics}

The quality control of photovoltaic systems has become a key application area for recent developments and extensions of the acceptance sampling methodology, although the results can certainly be adopted to many other areas in industry. 

The production of solar panels has become a highly complex high-throughput production process. Today's cell technologies rely on sophisticated solar cell designs. A solar cell can be regarded as a stack of thin layers, in order to trap as many photons as possible, transform them into electron-hole pairs and then ease the electrons' movement through the cell to the wires. Anti-reflective coatings of the glass covering have been introduced recently, in order to maximize the amount of sunlight trapped by the solar cells by channelling the photons to the lower layers of the cell. 
Optical filters are used in order to ensure that only those wavelengths pass that can be processed by the semiconductor to form electron-hole pairs. Let us briefly recall how a solar cell works:
The $p$-type silicon layer consists of silicon, which has four electrons, doped with a compound (such as Phosphorous) that contains one more valence electron, such that this layer is positively charged. The $n$-type silicon layer is silicon doped with compounds (such as Boron) having one less valence electron than silicon, such that only three electrons are available for binding with four adjacent silicon atoms. Thus, the $n$-type layer is negatively charged. An incomplete bond (hole) of the $n$-type layer can attract an electron to fill the hole, in which case the hole moves. At the $np$-junction where both layers meet, electrons from the $n$-type layer being freed by the photons' energy move to the $p$-layer and from there to the back contact, and the corresponding holes move to the contact grid at the top of the cell. This result in a current $I$. Combined with the internal electrical field of the cell due to the differently charged $p$- and $n$-silicon layers leads to power ($P=U I $), which can be used by an external load attached to the cell. 

Each layer of a cell makes use of specific physical and chemical properties of the base material and the added compounds. In a multijunction (tandem) design two cells are mechanically stacked on top of each other. The second cell at the bottom absorbs the higher energy photons not absorbed by the top cell. Such designs can increase
the efficiency  substantially. The physical and chemical interaction of those materials and particles is a complex dynamic process driven by the sun's irradiance, the associated heat and the weather conditions that may range from extreme cold to extreme heat.  Even in the absence of manufacturing faults, these facts cause serious changes of the physical and chemical properties due to ageing, and, as a consequence, also of the electrical properties of a solar cell, leading to what could be called degradation by design. 

Manufacturing faults add to those unavoidable sources of degradation. For example, even minor defects in the encapsulation of a PV module may result in leakage  after a couple of years, thus leading to internal corrosion and other effects that degrade or even destroy the module. Micro cracks in the crystalline semiconductor arising by improper handling of the modules at the production line, stress during transport to the site of construction or improper handling during assembly of the photovoltaic system, are invisible by eye but visible in electro luminescence (EL) images. Usually they have no effect on the electrical characteristics. However, it is conjectured that such micro cracks could have serious impact on
long-term degradation and result in failures after several years of operation. As a consequence, insurance companies routinely make EL images from insured lots or systems and are highly interested in the long-term influence of micro cracks.  Antireflective coatings have been shown experimentally to degrade after damp-heat tests leading to a loss of power. Driven by the high potential between the cell's surface and the ground, a likely source of potential induced degradation (PID) is the wandering of $Na^+ $ ions from the glass surface trough the cell to the $np$-junction, where they short-circuit the emitter. The fact that the emerging markets for photovoltaics are in countries such as India or Saudi Arabia, the degradation of the glass surface due to sand is an important issue for the reliability of PV systems.

As a consequence, there is a need for proper inspection plans that combine available information from the production line, quality assessments and audits at the time of delivery and construction of the solar systems and later inspections.

\section{Method}\label{t19:sec: method}

\subsection{Two-stage acceptance sampling}

To proceed, let us fix some notions more rigorously.
Let $ T_n $ be a statistic (decision function) depending on a sample $ X_1, \dots, X_n $ constructed in such a way that large values of $ \overline{X}_n - \tau $ indicate that the lot should be accepted. A pair $ (n,c)  \in \mathbb{N} \times [0,\infty) $ is called a {\em (acceptance) sampling plan}, if one draws a sample of $n$ observations and accepts the lot if $ \overline{X}_n - \tau > c $.
Then the probability that the lot is accepted,
\[
  OC(p) = P( T_n > c | p ), \qquad p \in [0,1],
\]
is called {\em operating characteristic}. Here $ P( \bullet | p ) $ indicates that the probability is calculated under the assumption that the true fraction of non-conforming items equals $p$.
Given specifications of the AQL and RQL and error probabilities $ \alpha $ and $ \beta $, a sampling plan is called {\em valid}, if 
\begin{equation}
\label{t19:AQLkrit}
  OC(p) \ge 1-\alpha, \qquad \text{for all $ p \le AQL $},
\end{equation}
and
\begin{equation}
\label{t19:RQLkrit}
  OC(p) \le \beta, \qquad \text{for all $ p \ge RQL $}.
\end{equation}

In this article, we consider two-stage acceptance procedures where a lot is examined at two time points. At time $ t_1 $, usually the time of production, delivery or construction of the system that uses the delivered items (PV modules), a control sample is taken, in order to decide whether the lot or shipment can be accepted. If the lot is rejected, we stop. If the lot is accepted, one proceeds and at time instant $ t_2 $ the system is inspected again. One applies a further acceptance sampling plan, based on a an inspection sample, in order to conclude whether the shipment is still in agreement with the specifications.

Let us denote the test statistic used at time $t_i $ using a sampling plan $ (n_i, c_i ) $ by $ T_{ni} = T_{n_i,i} $, $ i = 1, 2$. Notice that here and in what follows, with some abuse of usual notation, we indicate the dependence on $ n_i $ by $n$, in order to keep notation simple and clean; this will cause neither confusion nor conflict.

Then the corresponding operating characteristics are given by
\[
  OC_1(p) = P( T_{n1} > c_1 | p ), \qquad p \in [0,1],
\]
and, since the sampling plan $ (n_2, c_2) $ is constructed given $ T_{n1} > c_1 $,
\[
  OC_2(p) = P( T_{n2} > c_2 | T_{n1} > c_1, p), \qquad p \in [0,1].
\]
Since a lot is accepted if and only if it is accepted at stage $1$ and stage $2$, the 
{\em overall operating characteristic}, $OC(p) = P(\text{'lot acceptance'}|p) $, is given by
\begin{equation}
\label{t19:OverallOC}
  OC(p) = OC_1(p) OC_2(p), \qquad p \in [0,1].
\end{equation}
Of course, one may design the procedure such that at both stages the operating characteristics are valid for the same error probabilities. However, then one cannot
control the error probabilities of the overall procedure, since its operating characteristic is given by (\ref{t19:OverallOC}).

Thus, we propose to design the procedure by controlling the overall operating characteristics. This means, the stage-wise sampling plans are determined
in such a way that both of them ensure (\ref{t19:AQLkrit}) and (\ref{t19:RQLkrit}) for stage-specific error probabilities $ \alpha_1, \beta_1 $ and $ \alpha_2, \beta_2 $, i.e.
\begin{equation}
\label{t19:OCi1}
  OC_i(p) \ge 1 - \alpha_i, \qquad p \le AQL, \qquad (i=1,2),
\end{equation}
and
\begin{equation}
\label{t19:OCi2}
  OC_i(p) \le \beta_i, \qquad p \le RQL, \qquad (i=1,2).
\end{equation}
If (\ref{t19:OCi1}) and (\ref{t19:OCi2}) can be ensured with equality for $ p \in \{ AQL, RQL \} $, then we obtain
\[
  OC(AQL) = (1-\alpha_1)(1-\alpha_2), \qquad OC(RQL) = \beta_1 \beta_2
\]
for the overall OC curve. If we want that it represents an (overall) valid acceptance sampling plan, i.e.
\begin{equation}
\label{OCtwoPoint}
  OC(p) \ge 1-\alpha, \ p \le AQL, \quad OC(p) \le \beta, \ p \ge RQL,
\end{equation}
for given global error probabilities $ \alpha $ and $ \beta $,
we have to design the procedures at both stages appropriately, preferably such that (\ref{OCtwoPoint}) holds.
Treating the producer risk and the consumer risk symmetrically imposes the constraints
\[
  \alpha_1 = \beta_1, \qquad \alpha_2 = \beta_2,
\]
such that it remains to select $ \alpha_1 $ and $ \alpha_2 $ in such a way that
the resulting procedures guarantees a valid overall sampling plan.
For example, if one additionally imposes the constraint $ \alpha_1 = \alpha_2 $, one obtains a valid (overall) acceptance sampling plan, if $ (\alpha_1, \beta_1 ) $ is selected such that
$
  (1-\alpha_1)^2 = 1-\alpha, \beta_1^2 = \beta.
$
If now the global error probabilities $ \alpha $ and $ \beta $ are given and one puts $ \alpha_2 = \alpha_1 = 1 - \sqrt{1-\alpha} $ and $ \beta_1 = \beta_2 = \sqrt{\beta} $, then (\ref{OCtwoPoint})  holds with equalities, but the stage-wise consumer risks $ \beta_1, \beta_2 $ may be too high in practice -- observe that $ \sqrt{0.1} \approx 0.31623 $. Another approach is to use plans with $ \alpha_1 < \alpha_2 $, such that the procedure is, in terms of the producer risk, more restrictive at the first stage than at the second stage. In the simulations, we specified $ \alpha (= 0.1) $ and $ \alpha_1 $, determined the corresponding $ \alpha_2 $, i.e. $  \alpha_2 = 1 - \frac{1-\alpha}{1-\alpha_1},$ and put $ \beta_i = \alpha_i, i = 1, 2, $.
This means, at each stage producer and consumer risks are symmetric. As a consequence, the procedure will work on a small global consumer risk, since $ \beta = \beta_1 \beta_2 $ with $ \beta_1, \beta_2 \le \alpha_1 $. For example, the choices $ \alpha =0.1, \alpha_1 = 0.03 $ lead to $ \alpha_2 = \beta_2 \approx 0.072 $ such that $ \beta \approx 0.00216 $, yielding a valid acceptance sampling plan

\subsection{A two-stage procedure using additional data}

We assume that we are given an additional data set of size $m$, usually quite large, consisting of independent and identically distributed measurements, 
\[
  X_0, X_{01}, \dots, X_{0m} \stackrel{i.i.d.}{\sim} F_0,
\]
sampled at time $ t_0 < t_1 $, with 
\[
  \text{mean $ \mu_0 = E( X_0 ) $ and variance $ \sigma_0^2 = \text{Var}(X_0) $,}
\]
which can be used for the construction of the decision procedures. Recalling from our above discussion that those additional measurements usually represent historic data or are taken using a different measurement system, we allow for difference in location with respect to the independent and identically distributed control measurements,
\[
  X_1, X_{11}, \dots, X_{1n_1} \stackrel{i.i.d.}{\sim} F_1,
\]
taken at time instant $ t_1 $.  At the inspection time point $ t_2 $, we have additional measurements
\[
  X_2, X_{21}, \dots, X_{2n_2} \stackrel{i.i.d.}{\sim} F_2.
\]
Here we shall allow for a degradation effect leading to smaller measurements. 
Concretely, our distributional model relating the marginal distributions of the three samples is now as follows. 
\begin{equation}
  X_{ji} \sim
  \left\{
  \begin{array}{ll}
  F\left( \bullet - \Delta \right), & \qquad j = 0,\\
  F( \bullet ), & \qquad j = 1, \\
  F\left( \frac{ \bullet }{ d } \right), & \qquad j = 2.
  \end{array}
  \right.
\end{equation}
for constants $ \Delta \in \mathbb{R} $ and $ 0 < d < \infty $.
Equivalently, in terms of equality in distribution, 
\begin{align*}
  X_0 & \stackrel{d}{=} \Delta + X_1, \\
  X_1 & \sim F, \\
  X_2 & \stackrel{d}{=} d X_1. 
\end{align*}
The constant $d$ determines the degree of degradation (if $ d < 1 $) and is assumed to be known. We work with a simple degradation model, since in photovoltaics there is not yet enough knowledge about the degradation of photovoltaic modules, which would justify to go beyond the assumption that degradation acts like a damping factor on the power output measurement. We also assume that $d$ is known, since even the estimation of the mean  yearly degradation is a difficult practical problem and requires large data sets over long time periods, such it would be not realistic to estimate $d$ within our framework.

Further, we may and shall assume that 
\[
  F(x) =  G\left( \frac{ x - \mu }{\sigma} \right), \qquad x \in \mathbb{R},
\]
for some fixed but unknown d.f. $G$ with $ \int x dG(x) = 0 $ and $ \int x^2 d G(x) = 1 $, such that
\[
  \mu = E(X_1), \qquad \sigma^2 = \operatorname{Var}( X_1 ),
\]
are the mean and the variance of the control measurements taken at time $ t_1 $.

The two-stage acceptance sampling procedure to be studied is now as follows.
At stage $1$, i.e. at time $ t_1 $, based on a sampling plan $ (n_1, c_1) $ we accept the lot or shipment, if
\begin{equation}
\label{Decision1}
  T_{n1} = \sqrt{n_1} \frac{ \overline{X}_1 - \tau }{ S_m } > c_1,
\end{equation}
where 
$
  \overline{X}_1 = \frac{1}{n_1} \sum_{i=1}^{n_1} X_{1i}
$
is the average of the observations taken at time $ t_1 $ and
\[
  S_m = \sqrt{ \frac{1}{m-1} \sum_{i=1}^m (X_{0i} - \overline{X}_0)^2 }
\]
is the sample standard deviation calculated from the observation taken 
at time $ t_0 $. It is worth mentioning that standardizing with the sample standard deviation calculated from the time $ t_0 $ measurements is crucial; indeed, $ S_m $ can not be replace by, say, $ \widehat{\sigma}_1 = (n_1-1)^{-1} \sum_{j=1}^{n_i} (X_{1j} - \overline{X}_i)^2 $. 

If the lot is accepted at time $ t_1 $, we draw the additional
observations $ X_{21}, \dots, X_{2n_2} $ for inspection and calculate the
corresponding statistic
\[
  T_{n2} = \sqrt{n_2} \frac{ D \overline{X}_2 - \tau }{ S_m },
\]
where $ D = 1/d $ and $ \overline{X}_2 = \frac{1}{n_2} \sum_{i=1}^{n_2} X_{2i} $.
At inspection time $ t_2 $ the lot is accepted if
\begin{equation}
\label{Decision2}
  T_{n1} + T_{n2} > c_2.
\end{equation}
Notice that here we aggregate the available information by summing up $ T_{n1} $ and $ T_{n2} $. The rationale behind this rule is as follows. We reach the inspection time, only if
we passed the quality control at time $ t_1 $. The value of the statistic $ T_{n1} $ comprises the evidence in favor of acceptance and rejection, respectively. But even if the lot is accepted, the decision could be a close one, i.e. $ T_{n1} > c_ 1 $ but $ T_{n1} \approx c_1 $.  In such cases, the probability is relatively large that the lot is accepted at the inspection time again, if one drops the information already obtained at stage 1. Thus,  it makes sense to aggregate all available information to come to a decision, i.e. to take the sum of the statistics $ T_{n1} $ and $ T_{n2} $ and compare with a new critical value $c_2$.

\section{Approximations of the Operating Characteristics}\label{t19:sec: approx}

In order to calculate concrete acceptance sampling plans, we need to calculate the true operating characteristics, which is impossible without knowing the true underlying distributions. Thus, we shall derive appropriate approximations of the operating characteristics that will allow us to construct asymptotically optimal sampling plans.

Let us introduce the following notations. The standardized arithmetic averages will be denoted by
\[
  \overline{X}_i^* = \sqrt{n_i} \frac{ \overline{X}_i - \mu }{ \sigma }, \qquad i = 1, 2.
\]
Here and in what follows, we assume that $ D = 1 $, since otherwise one may replace the $ X_{2i} $ by $ D X_{2i} $.
It turns out that the asymptotically optimal acceptance sampling plans depend on the quantile function $ G^{-1}(p) $ of the standardized observations
\[ X_i^* = (X_i - \mu)/\sigma, \qquad i = 0, 1.\] 
Thus, we shall assume that we have an {\em arbitrary} consistent quantile estimator $ G_m^{-1}(p) $ of $ G^{-1}(p) $ at our disposal. It will be calculated from the additional sample taken at time $ t_0 $. We only need the following regularity assumption.

\textbf{Assumption Q:} One of the following two conditions holds.
\begin{itemize}
\item[(i)] $ G_m^{-1}(p) $ is a quantile estimator of the quantile function $ G^{-1}(p) $
of the standardized measurements satisfying the central limit theorem
\[
  \sqrt{m}( G_m^{-1}(p) - G^{-1}(p) ) \stackrel{d}{\to} V,
\]
as $ m \to \infty $, for some random variable $V$.
\item[(ii)] $ F_m^{-1}(p) $ is a quantile estimator of the quantile function $ F_0^{-1}(p) $ of the measurements taken at time $ t_0 $, satisfying the central limit theorem
\[
  \sqrt{m}( F_m^{-1}(p) - F_0^{-1}(p)) \stackrel{d}{\to} U,
\]
as $ m\to \infty $, for some random variable $U$.
\end{itemize}

\begin{remark}
\label{t19:RemarkQuantEst}
Notice that under condition (ii) of Assumption Q, one may construct a quantile estimator for $ G^{-1}(p) $ by
\[
  \widetilde{G}_m^{-1}(p) :=  \frac{ F_m^{-1}(p) - \mu_0 }{\sigma_0},
\]
if $ \mu_0 $ and $ \sigma_0 $ are known, since the quantiles of $F $ and $G$ are related by \[ F^{-1}(p) = \mu + \sigma G^{-1}(p). \] 
If $ \mu $ and $ \sigma $ are unknown, one should take the estimator
\begin{equation}
\label{t19:DefGM}
  G_m^{-1}(p) := \frac{F_m^{-1}(p) - \overline{X}_0}{S_m},
\end{equation}
where $ (\overline{X}_0,S_m) $ consistently estimates $ (\mu_0, \sigma_0) $ under our assumption of an i.i.d. sample $ X_{01}, \dots, X_{0m} $ a with finite second moment.
\end{remark}

Let us consider two examples.

\begin{example} A natural candidate for $ F_m^{-1}(p) $ is the corresponding sample quantile, 
\[ 
  \widetilde{F}_m^{-1}(p) = X_{0,( \lceil m p \rceil )}, \qquad p \in (0,1),
\]
where $ X_{0,(1)} \le \cdots \le X_{0,(m)} $ is the order statistic associated to $ X_{01}, \dots, X_{0m} $. However, it is known that the sample quantiles perform poorly for the type of acceptance sampling plans to be studied here, see \citet{MeisenPepelyshevSteland2012}.
\end{example}

\begin{example} Suppose that the distribution of the measurements is concentrated on a finite interval $ [a,b] $ that can be assumed to be $ [0,1] $. The Bernstein-Durrmeyer polynomial estimator of degree $N \in \mathbb{N} $ for $ F_0^{-1}(p) $ is then defined as 
\[
  F_{m,N}^{-1}(p) = (N+1) \sum_{i=0}^N a_i B_i^{(N)}(p), \qquad p \in (0,1),
\]
with coefficients
\[
  a_i = \int_0^1 \widetilde{F}_m^{-1}(u) B_i^N(u) \, du, \quad
\text{where} \quad
  B_i^{(N)}(x) =  {N \choose i} x^i (1-x)^{N-i}
\]
for $ i = 0, \dots, N $ are the Bernstein polynomials. In \citet{PepelyshevRafajlowiczSteland2013} it has been shown that $ \widehat{F}_{m,N}^{-1}(p) $ is consistent in the MSE and MISE sense and almost attains the optimal parametric rate of convergence if $ F_0^{-1} $ is smooth. The degree $N$ can be chosen in a data-adaptive way by controlling the number of modes of the density associated to the estimator as well as the closeness of the associated estimator of the distribution function $ \widehat{F}_{m,N}(x) $ in the sense that the maximal distance between $ \widehat{F}_{m,N}( \widehat{F}_{m,\widehat{N}}^{-1}(x) ) $ and the identity function $ i(x) = x $ is uniformly less or equal to $ 1/R_m $ where 
$ R_m = 2 \sqrt{m} / \sqrt{2 \log \log m} $; for details of the algorithm leading to the estimate $ \widehat{N} $ we refer to \citet{PepelyshevRafajlowiczSteland2013}. For the resulting estimate $ \widehat{F}_{m,\widehat{N}}^{-1}(p) $ an uniform error bound can be established,
\[
  \sup_q | \widehat{F}_{m,\widehat{N}}^{-1}(q) - F_0^{-1}(q) | \le 
  2 \sqrt{2 \log \log m} / \sqrt{m}, 
\]
see \citet[Theorem~3.1]{PepelyshevRafajlowiczSteland2013}. 
\end{example}

\begin{example} From previous studies it is known that quantile estimators obtained by (numerically) inverting a kernel density estimator 
\[ 
  \widehat{f}_m(x) = \frac{1}{m h} \sum_{i=1}^n K_h( x - X_{0i} ), \qquad x \in \mathbb{R},
\] 
provide accurate results for sampling plan estimation, see \cite{HerrmannSteland2010}. Here $K( \bullet ) $ is a regular kernel, usually chosen as a density with mean $0$ and unit variance, $ h > 0 $ the bandwidth and $ K_h(z) = K(z/h)/h $, $ z \in \mathbb{R} $, its rescaled version. The associated quantile estimator is obtained by solving for given $ p \in (0,1) $ the nonlinear equation
\[
  F_m( x_p ) = \int_{-\infty}^{x_p} \widehat{f}_m(x) \, dx \stackrel{!}{=} p.
\]
For a kernel density estimator one has to select the bandwidth $h$. If the resulting estimate is consistent for $ f(x) $ for each $ x \in \mathbb{R} $, which requires
to select $h = h_n $ such that $ h \to 0 $ and $ nh \to \infty $,
it follows that the  corresponding estimator of the distribution function,
$ F_m(x) = \int_{-\infty}^x \widehat{f}_m(u) \, du $, $ x \in \mathbb{R} $, is consistent as well, see \citet{Glick1974}, if the kernel used for smoothing is a density function.
\end{example}

In many cases, the central limit theorem for a quantile estimator, say, $ \widehat{Q}_m(p) $, of a quantile function $ F^{-1}(p) $ can be strengthened to a functional version that considers the scaled difference $ \sqrt{m}(\widehat{Q}_m(p)  - F^{-1}(p) ) $ as a function of $ p $.

\textbf{Assumption Q':} Assume that $ F_m^{-1}(p) $ is a quantile estimator of the quantile function $ F_0^{-1}( p ) $, $ 0 < p < 1 $, satisfying
\[
  \sqrt{m}( F_m^{-1}(p) - F_0^{-1}(p) ) \stackrel{d}{\to} \mathcal{F}(p),
\]
as $ m \to \infty $, for some random process $ \{ \mathcal{F}(p) : 0 < p < 1 \} $,
in the sense of weak convergence of such stochastic processes indexed by the unit interval. 

Assumption Q' holds true for the sample quantiles $ F_m^{-1}(p) $ as well as, for example, the Bernstein-Durrmeyer estimator, if the underlying distribution function attains a positive density. For the latter results, further details and discussion, we refer to \citet{PepelyshevRafajlowiczSteland2013}. 


Having defined the decision rules for acceptance and rejection at both stages
by (\ref{Decision1}) and (\ref{Decision2}), the overall OC curve is obviously given by
\[
  OC(p) = P( T_{n1} > c_1, T_{n1} + T_{n2} > c_2 | p ), \qquad p \in [0,1].
\]
Further, the operating characteristic of stage 2 is a conditional one and given by
\[
  OC_2(p) = P( T_{n1} + T_{n2} > c_2 | T_{n1} > c_1, p ),  \qquad p \in [0,1]. 
\]
Those probabilities cannot be calculated explicitly under the general assumptions of this paper. Hence, we need appropriate approximations in order to construct valid sampling plans. 

\subsection{Independent Sampling}

The case of independent sampling is, of course, of some relevance. In particular, it covers the case of destructive testing or, more generally, testing methods that may change the properties. Examples are accelerated heat-damp tests of PV modules. 
Let us assume that the samples $ X_{i1}, \dots, X_{in_i} $, $ i =1, 2 $, are independent.
Our OC curve approximations are based on the following result, which provides expansions of the test statistics involving the quantile estimates $ G_m^{-1}(p) $.

\begin{proposition} 
\label{t19:ExpansionIndependent}
Under independent sampling, we have
\[
  T_{n1} = \overline{X}_1^* - \sqrt{n_1} G_m^{-1}(p) + o_P(1),
\]
as $ n_1 \to \infty $ and $ m/n_1 = o(1) $, and
\[
 T_{n2} = \overline{X}_2^* - \sqrt{n_2} G_m^{-1}(p) + o_P(1),
\]
as $ n_2 \to \infty $ and $ m/n_2 = o(1) $. If a quantile estimator
$ F_m^{-1}(p) $ of $ F_0^{-1}(p) $ is available, both expansions
still hold true with $ G_m^{-1}(p) $ defined by (\ref{t19:DefGM}).
\end{proposition}

In what follows, $ \Phi(x) $ denotes the distribution function of the
standard normal distribution. We obtain the following approximation of the overall OC curve. The approximation holds in the following sense: 
We say $ A $ approximates $ A_n $ and  write  $ A_n \approx A $, if $ A_n = A + o_P(1) $, as $ \min(n_1,n_2) \to \infty $.

\begin{theorem} 
\label{t19:ApproxOCIndependent}
Under independent sampling and Assumption Q we have
\[
  OC_2( p ) \approx
  \frac{1}{\sqrt{2\pi}} 
  \frac{ \int_{c_1+\sqrt{n_1}G_m^{-1}(p) } 
  \left[ 1 - \Phi( c_2 - z + (\sqrt{n_1} + \sqrt{n_2}) G_m^{-1}(p) ) \right] e^{-z^2/2} \, dz }
  { 1- \Phi( c_1 + \sqrt{n_1} G_m^{-1}(p) ) },
\]
for any fixed $ p \in (0,1) $.
\end{theorem}

\subsection{Dependent Sampling}

If it is not necessary to rely on independent samples for quality control at time $ t_0 $ and inspection at time $ t_1 $, i.e. to test different modules at inspection, it is better to take the same modules. This means, one should rely on a panel design, where at time
$ t_0 $ or $ t_1 $ a random sample from the lot is drawn, the so-called panel, and that panel is also  analyzed at the time of inspection, i.e. the modules are remeasured. 
To simplify the technical proofs, we shall assume in the sequel that the panel is
established at time $ t_1 $ and that the sample taken at time $ t_0 $ is independent from the observations taken at later time points.

The control-inspection scheme studied in this paper aims at the minimization of the costs of testing by aggregating available information. Therefore, the inspection sample should be (and will be) considerably smaller than the first stage sample, i.e. $ n_2 << n_1 $, although it also may happen occasionally that $ n_2 > n_1 $, since the sample sizes are random.

In order to deal with this issue, the following dependent sampling scheme is proposed.
 If $ n_2 < n_1 $, one draws a subsample of size $ n_2 $ from the items drawn at time $ t_1 $ to obtain the control sample of size $ n_2 $. Those $n_2 $ items are remeasured at time instant $t_2 $ yielding the sample $ X_{21}, \dots, X_{2n_2} $. Notice that for fixed $i$ the measurements $ X_{1i} $ and $ X_{2i} $ are dependent, since they are obtained from the same item (module). Thus, we are given a paired sample
\[
  (X_{1i}, X_{2i}), \qquad i = 1, \dots, n_2,
\]
which has to be taken into account. 

It remains to discuss how to proceed, if  $ n_2 > n_1 $. Then one remeasures all $ n_1 $ items already sampled at time $ t_1 $ yielding $ n_1 $ paired observations $ (X_{1i}, X_{2i}) $, $ i = 1, \dots, n_1 $, and draws $ n_2-n_1 $ additional items from the lot. 

As a consequence, $ n_1 $ observations from the stage 2 sample are stochastically dependent from the stage 1 observations, whereas the others are independent. 
In order to proceed, let us assume that the sample sizes $ n_1 $ and $ n_2 $ satisfy
\begin{equation}
\label{RatioCondition}
  \lim \frac{n_1}{n_2} = \lambda.
\end{equation}
Notice that
\begin{align*}
    \operatorname{Cov}( \overline{X}_1^*, \overline{X}_2^* )
    & = \frac{1}{\sqrt{n_1n_2}} \sum_{i=1}^{n_1} \sum_{j=1}^{n_2} \operatorname{Cor}( X_{1i}, X_{2j} ) 
    = \sqrt{ \frac{n_1}{n_2} } \rho',
\end{align*}
where
\[
  \rho' = \operatorname{Cor}( X_1, X_2 ) \not = 0.
\]
Thus, if $ \rho' \not= 0 $, the approximation results of the previous subsection are no longer valid, since even asymptotically $ \overline{X}_1 $ and $ \overline{X}_2 $ are correlated, and thus the standardized versions are correlated as well under this condition.

The following results provide the extensions required to handle this case of dependent sampling. Proposition~\ref{t19:CLTDependent1} provides the asymptotic normality of the sample averages, which share the possibly non-trivial covariance. 

\begin{proposition}
\label{t19:CLTDependent1} Suppose that the above sampling scheme at stages 1 and 2 is applied and assume that one of the following assumptions is satisfied.
\begin{itemize}
  \item[(i)]  $ X_{01}, \dots, X_{0m} $ is an i.i.d. sample with common distribution function $ F_0(x) = F(x-\Delta) $ and Assumption Q holds.
  \item[(ii)]  Assumption Q' is satisfied.
  \end{itemize}
  Then we have
  \[ 
    \left( 
      \begin{array}{cc} \overline{X}_1^* \\ \overline{X}_2^* \end{array}
    \right)
    \stackrel{d}{\to} N\left( \left(\begin{array}{cc} 0 \\ 0 \end{array}\right), \bfSigma \right),
  \]
  as $ \min(n_1, n_2) \to \infty $ with $ n_1/n_2 \to \lambda $, where the asymptotic
  covariance matrix is given by
  \[
    \bfSigma = \left( \begin{array}{cc} 1 & \rho \\ \rho & 1 \end{array} \right),
  \]
  with $ \rho = \sqrt{\lambda} \operatorname{Cor}( X_1, X_2 ) $.
\end{proposition}

The following theorem now establishes expansions of the test statistics, which hold jointly.

\begin{theorem} 
\label{t19:ExpansionDependent1}
Suppose that the above sampling scheme at stages 1 and 2 is applied and assume that one of the following assumptions is satisfied.
\begin{itemize}
  \item[(i)]  $ X_{01}, \dots, X_{0m} $ is an i.i.d. sample with common distribution function $ F_0(x) = F(x-\Delta) $  and Assumption Q holds.
  \item[(ii)]  Assumption Q' is satisfied.
  \end{itemize}
  Then we have
\[
  \left( \begin{array}{cc} T_{n1} \\ T_{n2} \end{array} \right)
  =
    \left( 
      \begin{array}{cc} \overline{X}_1^* \\ \overline{X}_2^* \end{array}
    \right)
  -
  \left( \begin{array}{cc} \sqrt{n_1} G_m^{-1}(p) \\ \sqrt{n_2} G_m^{-1}(p) \end{array} \right)
  + o_P(1),
\]
as $ \min(n_1,n_2) \to \infty $ with $ n_1/n_2 \to \infty $ and $ \max(n_1,n_2)/m = o(1) $.
\end{theorem}

The approximation of the OC curve $ OC_2(p) $ is now more involved. Recall at this point the well known fact that for a random vector $ (X,Y) $ that is bivariate normal with mean vector $ (\mu_X, \mu_Y)' $, variances $ \sigma_X^2 = \sigma_Y^2 = 1 $ and correlation $ \rho_{XY} $, the conditional distribution of, say, $ Y $ given $ X = z $ attains the Gaussian density
\[
  x \mapsto \frac{1}{\sqrt{2\pi (1-\rho^2)}} 
  \exp\left( 
    - \frac{ (x - \rho z)^2 }{2 \sqrt{1-\rho_{XY}^2}} \right), \qquad x \in \mathbb{R}.
\]

The following theorem now provides us with the required approximation of
the operating characteristic for the second stage sampling plan. 
It will be established in the following sense, slightly modified compared to the previous subsection: We say $ A $ approximates $ A_n $ and 
write  $ A_n \approx A $, if $ A_n = A + o_P(1) $, as $ \min(n_1,n_2) \to \infty $ with $ n_1/n_2 \to \lambda $, $ \max(n_1,n_2)/m = o(1) $ and $ n_1 \ge n \to \infty $.

\begin{theorem}
\label{t19:ApproximationDependent1}
Suppose that the above sampling scheme at stages 1 and 2 is applied and assume that one of the following assumptions is satisfied.
\begin{itemize}
  \item[(i)]  $ X_{01}, \dots, X_{0m} $ is an i.i.d. with common distribution function $ F_0(x) $  and Assumption Q holds.
  \item[(i)]  Assumption Q' is satisfied.
  \end{itemize}
  If, additionally, $ |\rho| \le \overline{\rho} < 1 $, then we have
\begin{equation}
  OC_2(p) \approx
  \frac{1}{\sqrt{2\pi}}
  \frac{ \int_{c_1 + \sqrt{n_1} G_m^{-1}(p) } 
  \left[ 1 - \Phi\left(  \frac{ c_2 - z + (\sqrt{n_1} + \sqrt{n_2}) G_m^{-1}(p) - \widehat{\rho} z }{ \sqrt{ 1 - \widehat{\rho}^2 } } \right) \right] e^{-z^2/2} \, dz }
  { 1 - \Phi( c_1 + \sqrt{n_1} G_m^{-1}(p) ) },
\end{equation}
where
\[
  \widehat{\rho} = \sqrt{\frac{n_1}{n_2}} \frac{ \widehat{\gamma} }{ \widehat{\sigma}_1 \widehat{\sigma}_2 } 
\ \ \text{with} \ \
  \widehat{\sigma}_j^2 = \frac{1}{n} \sum_{i=1}^n (X_{ji} - \overline{X}_i)^2,
  \ j = 1,2,
\]
and
\[
  \widehat{\gamma} = \frac1n \sum_{i=1}^n ( X_{1i} - \overline{X}_1)( X_{2i} - \overline{X}_2 ).
\]
\end{theorem}

The above result deserves some discussion.

\begin{remark}
Observe that the unknown correlation coefficient is estimated from $n$ pairs
$ (X_{1i}, X_{2i}) $, $ i = 1, \dots, n $. Since the sampling plan $ (n_2, c_2 ) $
cannot be determined without an estimator $ \widehat{\rho} $, on should
fix $ n \le n_1 $ and remeasure $ n $ items at inspection time $ t_2 $, in order
to estimate $ \rho $.
\end{remark}

\begin{remark}
The fact that the approximation also holds true under the general probabilistic assumption Q' points to the fact that the results generalize the acceptance sampling methodology to the case of dependent sampling, for example when it is not feasible to draw randomly from the lot and instead one has to rely on consecutive produced items that are very likely to be stochastically dependent due to the nature of the production process.
\end{remark}

\begin{remark} The condition (\ref{RatioCondition}) can be easily ensured by
replacing $ n_2 $ by $ n_1/\lambda $, i.e. put $ n_2(\lambda) = n_1/\lambda $ and determining $ \lambda $ such that a valid sampling plan $ (n_2,c_2) $ results. However,
the procedure is not reformulated in this way for sake of clarity.
\end{remark}

\subsection{Sampling in spatial batches}

In photovoltaic quality control, it is quite common to sample in spatial batches.
Here one selects randomly a solar panel from the photovoltaic system, usually
arranged as a grid spread out over a relatively large area. Then the selected module
and $b-1$ neighbouring modules are measured on site. Of course, observations from neighbouring modules are correlated, since they share various factors that affect
variables relevant for quality and reliability. Among those are the frame on which they are installed, so that they share risk factors due to wrong installation, the local climate within the area (wind and its direction that leads to stress due to vibrations, see \citet{AssmusEtAl2011}), the wires as well as the inverter to which they are connected. Further, one cannot assume that during installation the modules are randomly spread over the area, so that their ordering may be the same as on the production line.

So let us assume that one substitutes $ n_i $ by $ \lceil n_i/b \rceil b $ and $ c_i $ by
the re-adjusted critical value (see step 6 of the algorithm in Section~\ref{t19:sec: comp}. Thus, we may and will assume
that 
\[
  n_i =  r_i b, \qquad i = 1, 2,
\]
where $b$ is the batch size and $ r_i $ the number of randomly selected batches.
Suppose that the observations are arranged such that 
\[ 
  (X_{i1}, \dots, X_{in_i}) = (X_1^{(1)}, \dots, X_b^{(1)}, \dots, X_1^{(r)}, \dots, X_b^{(r)} ),
\] 
where $ X_\ell^{(j)} $ is the $\ell$th observation from batch $j$, $ \ell = 1, \dots, b $, $ j = 1, \dots, r_i $.

Let us assume the following spatial-temporal model:
\[
  X_{i,(\ell-1)b+j} = \mu_i + B_\ell + \epsilon_{ij},
\]
for $ i =1, 2 $, $\ell = 1, \dots, r $ and $ j = 1, \dots, b $. Here
$ \{ \epsilon_{ij} : 1 \le j \le b, i = 1,2 \} $ are i.i.d. $(0, \sigma_\epsilon^2) $ error
terms, $ \{ B_\ell : \ell = 1, \dots, r \} $ are i.i.d. $(0,\sigma_B^2) $ random variables
representing the batch effect. It is assumed that $ \{ \epsilon_{ij} \} $ and $ \{ B_\ell \} $ are independent.

Then the covariance matrix of the random vector $ \mathbf{X}_i = ( X_{i1}, \dots, X_{1n_i}) $ is given by
\[
  \operatorname{Cov}(\mathbf{X}_i) 
  =  \bigoplus_{i=1}^{r_i} [  \sigma_B^2 \textbf{J}_b 
  + \sigma_\epsilon^2 \textbf{I}_b ],
\]
for $ i =1,2 $, where $ \textbf{J}_b $ denotes the $(b \times b)$-matrix with entries $1$
and $ \textbf{I}_b $ is the $b$-dimensional identity matrix.
Observing that
\[
  \operatorname{Cov}( \sqrt{n_1} \overline{X}_1, \sqrt{n_2} \overline{X}_2 )
  =
  \frac{S}{\sqrt{n_1} \sqrt{n_2}} 
\]
where $S$ is the sum of all elements of $ \operatorname{Cov}( \mathbf{X}_1 ) $,
we obtain
\[
  \operatorname{Cov}( \sqrt{n_1} \overline{X}_1, \sqrt{n_2} \overline{X}_2 )
  =
  \frac{ r_1b^2 \sigma_B^2 + r_1b \sigma_\epsilon^2}{ 
  \sqrt{r_1 r_2} b }
  = \sqrt{ \frac{r_1}{r_2} } b \sigma_B^2 + \sqrt{ \frac{r_1}{r_2} } \sigma_\epsilon^2.
\]
It can be shown that the method of proof used to show the above results extends to that 
spatial batch sampling, if one additionally assumes that $b$ is fixed and
\[
  \lim \frac{r_1}{r_2} = r^* > 0.
\]

\section{Computational aspects}\label{t19:sec: comp}

It is worth discussing some computational aspects. We confine ourselves to
the case of independent sampling, since the modifications for the dependent
case are then straightforward.

The calculation of the two-stage sampling plan is now as follows. At stage 1,
one solves the  equations
\[
  OC_1( AQL ) = 1-\alpha_1, \qquad OC_1( RQL ) = \beta_1, 
\]
leading to the explicit solutions
\begin{align}
\label{t19:N1_Formula}
  n_1= \biggl\lceil \frac{ (\Phi^{-1}( \alpha_1 ) - \Phi^{-1}(1-\beta_1) )^2 }{ (G_m^{-1}(AQL) - G_m^{-1}(RQL) )^2 } \biggr\rceil, \\
  \label{t19:C1_Formula}
  c_1 = - \frac{\sqrt{n_1}}{2} ( G_m^{-1}(AQL) + G_m^{-1}(RQL) ).
\end{align}
The sampling plan $(n_2,c_2)$ for stage 2 has to be determine such that
\[
  OC_2(AQL) = 1-\alpha_2, \qquad OC_2(RQL) = \beta_1,
\]
which is done by replacing $ OC_2 $ by its approximation, thus leading us to the 
nonlinear equations
\[
    \frac{1}{\sqrt{2\pi}} 
  \frac{ \int_{c_1+\sqrt{n_1}G_m^{-1}(AQL) } 
  \left[ 1 - \Phi( c_2 - z + (\sqrt{n_1} + \sqrt{n_2}) G_m^{-1}(AQL) ) \right] e^{-z^2/2} \, dz }
  { 1- \Phi( c_1 + \sqrt{n_1} G_m^{-1}(AQL) ) }
 = 1-\alpha_2 
\]
and 
\[
  \frac{1}{\sqrt{2\pi}} 
  \frac{ \int_{c_1+\sqrt{n_1}G_m^{-1}(RQL) } 
  \left[ 1 - \Phi( c_2 - z + (\sqrt{n_1} + \sqrt{n_2}) G_m^{-1}(RQL) ) \right] e^{-z^2/2} \, dz }
  { 1- \Phi( c_1 + \sqrt{n_1} G_m^{-1}(RQL) ) }
  = \beta_2,
\]
which have to be solved numerically. Notice that the integrals appearing at the left side also have to be calculated numerically.

In order to calculate the sampling plan $ (n_2, c_2) $, the following straightforward algorithm performed well and was used in the simulation study.

\vskip 0.5cm
\noindent
{\sc Algorithm:}

\begin{enumerate}
  \item Select $ \varepsilon > 0 $.
   \item Calculate $ (n_1, c_1) $ using (\ref{t19:N1_Formula}) and (\ref{t19:C1_Formula}).
   \item Perform a grid search minimization of the OC curve
  over $ (n,c) \in \{ (n',c') : c' = 1, \dots, c^*(n'),\ n' = 1, \dots, 200 \} $, 
  where $ c^*(n') = \min \{ 1 \le c'' \le 60  : (OC(AQL)-(1-\alpha_2))^2 + (OC(RQL)-\beta_2)^2 \le \varepsilon \} $ for given $n' $. Denote the grid-minimizer by $ (n^*,c^*) $.
  \item Use the grid-minimizer $ (n^*, c^*) $ as a starting value for numerically
  solving the nonlinear equations up to an error bound $ \varepsilon $ for the
  sum of squared deviations from the target. Denote the minimizer by $ (n_2^*, c_2^*)$.
  \item Put $ n_2 = \lceil n_2 \rceil $.
  \item For fixed $n = n_2 $ minimize numerically the nonlinear equations with respect to $ c_2 $ up to an error bound  $ \varepsilon $ for the
  sum of squared deviations from the target. Denote the minimizer by $ c_2^* $.
  \item Output $ (n_2, c_2) = (n_2, c_2^*) $.
\end{enumerate}

It turned out that the combination of a grid search to obtain starting values and a 
two-pass successive invocation of a numerical optimizer to minimize with respect to the sample size and the control limit in the first stage and, after rounding up the sample size, minimizing with respect to the control limit results in a stable algorithm.

\section{Simulations}\label{t19:sec: sims}

The simulation study has been conducted, in order to get some insights into the final sample statistical properties of the procedures. It was designed to mimic certain distributional settings that are of relevance in photovoltaic quality control. 

It is known from previous studies that the standard deviation of the estimated sample size is often quite high even when a large data set $ X_{01}, \dots, X_{0m} $ can be used to estimate it, see \citet{MeisenPepelyshevSteland2012}, \citet{GolyandinaPepelyshevSteland2012} and \citet{PepelyshevRafajlowiczSteland2013}.
The question arises how accurately the second stage sampling plan can be estimated, having in mind that the estimated first stage sample size affects the operating characteristic at the second stage.

For the simulations the following parameters were used: $ \alpha = \beta = 0.1 $ (global error probabilities), $AQL = 2\% $ and $ RQL = 5\%$.  The error probabilities $ \alpha_1 = \beta_1 $ for the first stage acceptance sampling procedure were selected from the set $ \{ 0.03, 0.05, 0.07 \} $ and the corresponding value $ \alpha_2 = 1-(1-\alpha)/(1-\alpha_1) $ was then calculated for the second stage inspection, cf. our discussion in Section~\ref{t19:sec: method}. The sample size $ m $ of the additional sample from the production line was chosen as $ 250 $ and $ 500 $.

Data sets according to the following models were simulated:
		\begin{eqnarray}
			&\text{Model 1:}&\;\;X_0\sim F^1 = N(220,4), \nonumber \\ 
			&\text{Model 2:}&\;\;X_0\sim F^2 = 0.9N(220,4)+0.1N(230,8).\nonumber \\
			&\text{Model 3:}&\;\;X_0\sim F^3 = 0.2N(200,4)+0.6N(220,4)+0.2N(230,8).\nonumber\\
			&\text{Model 4:}&\;\;X_0\sim F^4 = 0.2N(212,4)+0.6N(220,8)+0.2N(228,6).\nonumber
		\end{eqnarray}

The required quantiles for methods based on the kernel density estimator for the construction of the sampling plans were estimated by numerically inverting an integrated kernel density estimator $ \widehat{f}_m(x) $ calculated from the standardized sample $ X_{01}^*, \dots, X_{0m}^* $. The following methods of quantile estimation were used, where the first four approaches employ the kernel estimator with different bandwidth selectors:
 \begin{enumerate}
 \item Biased cross-validated (BCV) bandwidth.
 \item Sheather-Johnson bandwidth selection (SJ), \citet{SheatherJones1991}.
 \item Golyandina-Pepeyshev-Steland method (GPS), \citet{GolyandinaPepelyshevSteland2012}.
 \item Indirect cross-validation (ICV), \citet{SavchukEtAl2010}.
 \item Bernstein-Durrmeyer polynomial (BDP) quantile estimator, \cite{PepelyshevRafajlowiczSteland2013}.
\end{enumerate}

The following tables summarize the simulation results. Each case was simulated using 10,000 repetitions.

Table~\ref{t19:TableModel1} provides results for normally distributed measurements with mean $220$ and variance $ 4 $. The results show that even for such small sample sizes as $ 250 $ and $ 500 $, respectively, the second-stage sampling plan $ (n_2, c_2) $ can be estimated with comparable accuracy as the first-stage plan.  Further,
it can be seen that the GPS bandwidth selector provides on average the smallest sampling plan numbers $ n_2 $ and the highest accuracy.

For Model 2, a mixture model where for 10\% of the items the mean is reduced by 10 units, the situation is now different. Here biased cross-validation and indirect cross-validation perform best and produce the most accurate estimates, see Table~\ref{t19:TableModel2}. Again, the stage-two plan can be estimated with comparable accuracy. 

Model 3 represents a symmetric distribution with two smaller subpopulations whose mean is larger or smaller, such that there are notable local minima of the density between the three corresponding local maxima. The results are given in Table~\ref{t19:TableModel3}. Whereas for Models 1, 2 and 4 the GDP method leads to larger expected sample sizes and larger standard deviations than the other methods, it outperforms all other methods under Model 3, when $ m = 250 $.

Of considerable interest in photovoltaic applications, and presumable other areas as well, is Model 4, a kind of head-and-shoulders distribution resulting in relatively short tails.
The results in Table~\ref{t19:TableModel4} demonstrate that in this case the GPS method
provides the best results in all cases, both in the sense of smallest expected sample sizes for both stages and in the sense of highest accuracy of estimation (i.e. smallest standard deviations). 

	\begin{table}[h]
	\caption{Characteristics of the sampling plans for Model 1}
	\label{t19:TableModel1}
	\begin{tabular}{p{0.8cm}p{1cm}p{0.6cm}p{1cm}p{0.9cm}p{0.9cm}p{0.8cm}p{0.8cm}p{0.8cm}p{0.8cm}p{0.8cm}p{0.8cm}}
\hline\noalign{\smallskip}
 $\alpha_1$ & $\alpha_2$ & $m$ & Type & $E(n_1)$ & $sd(n_1)$ & $c_1$ & $sd(c_1)$ & $E(n_2)$ & $sd(n_2)$ & $c_2$ & $sd(c_2)$ \\ \hline

 3\% & 7.22\% & 250 & BCV & 79.76 & 22.47 & 17.39 & 2.14 & 18.33 & 8.63 & 26.30 & 3.34 \\ 
        &  & 250 & SJ & 82.13 & 25.42 & 17.43 & 2.38 & 19.97 & 9.97 & 26.41 & 3.42 \\ 
       &   & 250 & GPS & 78.92 & 21.82 & 17.37 & 2.11 & 17.68 & 8.29 & 26.23 & 3.33 \\ 
       &  & 250 & ICV & 80.10 & 22.89 & 17.40 & 2.17 & 18.50 & 8.81 & 26.31 & 3.34 \\ 
       &  & 250 & BDP & 90.58 & 34.08 & 16.91 & 2.64 & 26.97 & 13.95 & 26.28 & 3.29 \\ 
 7\% & 3.23\% & 250 & BCV & 49.29 & 13.84 & 13.67 & 1.68 & 22.35 & 8.46 & 23.71 & 3.07 \\ 
         &             & 250 & SJ & 50.76 & 15.67 & 13.71 & 1.86 & 23.85 & 9.60 & 23.93 & 3.38 \\ 
    &  & 250 & GPS & 48.78 & 13.43 & 13.66 & 1.65 & 21.80 & 8.22 & 23.61 & 2.99 \\ 
    &  & 250 & ICV & 49.50 & 14.09 & 13.68 & 1.70 & 22.52 & 8.62 & 23.73 & 3.11 \\ 
   &  & 250 & BDP & 55.98 & 21.01 & 13.29 & 2.07 & 31.04 & 13.83 & 23.99 & 3.76 \\ 
   3\% & 7.22\% & 500 & BCV & 80.21 & 17.99 & 17.26 & 1.72 & 19.11 & 6.86 & 26.55 & 2.94 \\ 
    &  & 500 & SJ & 81.60 & 19.51 & 17.27 & 1.84 & 20.09 & 7.52 & 26.66 & 3.06 \\ 
    &  & 500 & GPS & 79.49 & 17.42 & 17.25 & 1.67 & 18.50 & 6.64 & 26.44 & 2.88 \\ 
    &  & 500 & ICV & 80.36 & 18.22 & 17.26 & 1.73 & 19.18 & 6.98 & 26.56 & 2.97 \\ 
    &  & 500 & BDP & 93.90 & 24.27 & 17.43 & 1.95 & 27.32 & 9.90 & 27.38 & 2.83 \\ 
   7\% & 3.23\% & 500 & BCV & 49.58 & 11.07 & 13.57 & 1.34 & 23.17 & 7.05 & 23.68 & 2.52 \\ 
   &  & 500 & SJ & 50.45 & 12.01 & 13.58 & 1.44 & 24.09 & 7.57 & 23.81 & 2.69 \\ 
  &  & 500 & GPS & 49.15 & 10.72 & 13.56 & 1.31 & 22.62 & 6.84 & 23.60 & 2.44 \\ 
  &  & 500 & ICV & 49.67 & 11.21 & 13.57 & 1.35 & 23.27 & 7.19 & 23.70 & 2.54 \\ 
  &  & 500 & BDP & 58.01 & 14.93 & 13.70 & 1.53 & 31.97 & 10.12 & 24.70 & 2.89 \\ 
  
\noalign{\smallskip}\hline\noalign{\smallskip}
	\end{tabular}
	\end{table}
	
	\begin{table}[h]
	\caption{Characteristics of the sampling plans for Model 2}
	\label{t19:TableModel2}
	\begin{tabular}{p{0.8cm}p{1cm}p{0.6cm}p{1cm}p{0.9cm}p{0.9cm}p{0.8cm}p{0.8cm}p{0.8cm}p{0.8cm}p{0.8cm}p{0.8cm}}
\hline\noalign{\smallskip}
 $\alpha_1$ & $\alpha_2$ & $m$ & Type & $E(n_1)$ & $sd(n_1)$ & $c_1$ & $sd(c_1)$ & $E(n_2)$ & $sd(n_2)$ & $c_2$ & $sd(c_2)$ \\  \hline

 3\% & 7.22\% & 250 & BCV & 281.52 & 88.60 & 21.80 & 2.45 & 115.68 & 46.15 & 29.34 & 2.28 \\ 
  &  & 250 & SJ & 296.12 & 94.82 & 22.07 & 2.60 & 126.44 & 45.39 & 29.92 & 2.71 \\ 
   &  & 250 & GPS & 297.60 & 97.68 & 22.06 & 2.63 & 128.48 & 48.30 & 30.38 & 2.82 \\ 
   &  & 250 & ICV & 274.44 & 83.41 & 21.67 & 2.33 & 111.32 & 39.88 & 30.22 & 2.50 \\ 
    &  & 250 & BDP & 320.28 & 123.21 & 21.35 & 3.20 & 136.48 & 57.84 & 29.55 & 1.97 \\ 
   7\% & 3.23\% & 250 & BCV & 173.44 & 54.63 & 17.11 & 1.93 & 110.88 & 38.94 & 31.64 & 3.62 \\ 
    &  & 250 & SJ & 182.56 & 58.42 & 17.33 & 2.04 & 118.28 & 38.59 & 32.22 & 3.66 \\ 
    &  & 250 & GPS & 183.32 & 60.02 & 17.32 & 2.06 & 122.88 & 42.15 & 32.34 & 3.74 \\ 
    &  & 250 & ICV & 169.20 & 51.33 & 17.01 & 1.83 & 108.20 & 33.66 & 31.42 & 3.30 \\ 
    &  & 250 & BDP & 197.52 & 75.82 & 16.77 & 2.51 & 129.72 & 50.85 & 31.20 & 4.59 \\ 
   3\% & 7.22\% & 500 & BCV & 280.24 & 56.48 & 21.94 & 2.04 & 116.00 & 26.61 & 28.84 & 2.82 \\ 
    &  & 500 & SJ & 289.00 & 61.28 & 22.14 & 2.19 & 122.44 & 30.35 & 28.62 & 2.08 \\ 
    &  & 500 & GPS & 283.32 & 62.64 & 22.02 & 2.14 & 118.40 & 30.42 & 28.49 & 2.09 \\ 
    &  & 500 & ICV & 276.88 & 53.26 & 21.86 & 1.99 & 114.52 & 23.33 & 28.37 & 2.13 \\ 
    &  & 500 & BDP & 331.44 & 91.83 & 22.40 & 2.73 & 138.20 & 42.49 & 30.36 & 3.29 \\ 
   7\% & 3.23\% & 500 & BCV & 172.84 & 34.83 & 17.23 & 1.60 & 110.52 & 21.92 & 31.89 & 2.96 \\ 
    &  & 500 & SJ & 178.12 & 37.56 & 17.38 & 1.71 & 115.16 & 25.58 & 32.14 & 3.19 \\ 
    &  & 500 & GPS & 174.68 & 38.62 & 17.29 & 1.68 & 113.04 & 29.44 & 31.68 & 2.84 \\ 
    &  & 500 & ICV & 170.68 & 32.78 & 17.16 & 1.56 & 108.96 & 20.69 & 31.78 & 2.89 \\ 
    &  & 500 & BDP & 204.36 & 56.50 & 17.59 & 2.14 & 133.32 & 35.27 & 32.30 & 3.34 \\ 
\noalign{\smallskip}\hline\noalign{\smallskip}
	\end{tabular}
	\end{table}

	\begin{table}[h]
	\caption{Characteristics of the sampling plans for Model 3}
	\label{t19:TableModel3}
	\begin{tabular}{p{0.8cm}p{1cm}p{0.6cm}p{1cm}p{0.9cm}p{0.9cm}p{0.8cm}p{0.8cm}p{0.8cm}p{0.8cm}p{0.8cm}p{0.8cm}}
\hline\noalign{\smallskip}
 $\alpha_1$ & $\alpha_2$ & $m$ & Type & $E(n_1)$ & $sd(n_1)$ & $c_1$ & $sd(c_1)$ & $E(n_2)$ & $sd(n_2)$ & $c_2$ & $sd(c_2)$ \\ \hline

 3\% & 7.22\% & 250 & BCV & 206.12 & 78.11 & 26.99 & 4.46 & 75.92 & 33.60 & 29.05 & 2.57 \\ 
    &  & 250 & SJ & 210.84 & 78.68 & 27.26 & 4.55 & 77.44 & 34.14 & 30.39 & 2.96 \\ 
    &  & 250 & GPS & 203.44 & 75.73 & 26.86 & 4.26 & 72.64 & 33.43 & 29.36 & 2.99 \\ 
    &  & 250 & ICV & 202.96 & 76.32 & 26.83 & 4.36 & 74.12 & 33.60 & 29.57 & 3.01 \\ 
    &  & 250 & BDP & 171.00 & 66.69 & 23.82 & 3.70 & 58.76 & 29.80 & 29.15 & 2.07 \\ 
   7\% & 3.23\% & 250 & BCV & 127.16 & 48.05 & 21.20 & 3.50 & 73.36 & 35.32 & 33.94 & 2.66 \\ 
    &  & 250 & SJ & 129.96 & 48.50 & 21.40 & 3.57 & 75.60 & 35.86 & 33.73 & 2.48 \\ 
    &  & 250 & GPS & 125.36 & 46.58 & 21.09 & 3.33 & 70.88 & 34.31 & 33.98 & 2.63 \\ 
    &  & 250 & ICV & 125.12 & 47.01 & 21.07 & 3.42 & 72.44 & 36.28 & 34.05 & 2.31 \\ 
    &  & 250 & BDP & 105.64 & 41.11 & 18.73 & 2.90 & 58.64 & 28.75 & 32.19 & 3.20 \\ 
   3\% & 7.22\% & 500 & BCV & 190.80 & 58.17 & 26.28 & 3.54 & 68.00 & 25.37 & 28.45 & 2.04 \\ 
    &  & 500 & SJ & 191.68 & 60.03 & 26.30 & 3.67 & 67.92 & 25.75 & 29.17 & 2.57 \\ 
    &  & 500 & GPS & 188.96 & 56.97 & 26.20 & 3.44 & 66.52 & 24.81 & 28.79 & 1.74 \\ 
    &  & 500 & ICV & 190.68 & 57.17 & 26.30 & 3.48 & 67.80 & 25.28 & 28.74 & 1.92 \\ 
    &  & 500 & BDP & 194.84 & 49.04 & 25.89 & 2.67 & 71.80 & 22.06 & 28.73 & 1.62 \\ 
   7\% & 3.23\% & 500 & BCV & 117.64 & 35.78 & 20.64 & 2.77 & 65.76 & 27.16 & 33.69 & 2.20 \\ 
    &  & 500 & SJ & 118.40 & 36.91 & 20.67 & 2.87 & 65.12 & 27.68 & 33.44 & 2.44 \\ 
    &  & 500 & GPS & 116.48 & 34.94 & 20.57 & 2.69 & 65.36 & 26.14 & 33.62 & 2.11 \\ 
    &  & 500 & ICV & 117.56 & 35.21 & 20.65 & 2.72 & 66.32 & 26.38 & 33.79 & 2.23 \\ 
    &  & 500 & BDP & 120.12 & 30.31 & 20.33 & 2.10 & 67.88 & 22.75 & 34.64 & 2.05 \\ 

\noalign{\smallskip}\hline\noalign{\smallskip}
	\end{tabular}
	\end{table}

	\begin{table}[h]
	\caption{Characteristics of the sampling plans for Model 4}
	\label{t19:TableModel4}
	\begin{tabular}{p{0.8cm}p{1cm}p{0.6cm}p{1cm}p{0.9cm}p{0.9cm}p{0.8cm}p{0.8cm}p{0.8cm}p{0.8cm}p{0.8cm}p{0.8cm}}
\hline\noalign{\smallskip}
 $\alpha_1$ & $\alpha_2$ & $m$ & Type & $E(n_1)$ & $sd(n_1)$ & $c_1$ & $sd(c_1)$ & $E(n_2)$ & $sd(n_2)$ & $c_2$ & $sd(c_2)$ \\ \hline

 3\% & 7.22\% & 250 & BCV & 171.76 & 49.49 & 24.52 & 2.58 & 59.79 & 23.18 & 28.55 & 1.85 \\ 
   &  & 250 & SJ & 230.81 & 65.47 & 27.29 & 3.28 & 87.28 & 29.17 & 29.45 & 1.91 \\ 
    &  & 250 & GPS & 173.57 & 46.82 & 24.62 & 2.49 & 60.70 & 22.25 & 28.57 & 1.74 \\ 
    &  & 250 & ICV & 173.32 & 44.49 & 24.61 & 2.37 & 60.56 & 21.12 & 28.47 & 1.71 \\ 
    &  & 250 & BDP & 251.94 & 80.19 & 26.98 & 3.67 & 96.91 & 34.33 & 29.63 & 2.04 \\ 
   7\% & 3.23\% & 250 & BCV & 105.94 & 30.47 & 19.26 & 2.03 & 58.85 & 24.02 & 33.12 & 2.37 \\ 
    &  & 250 & SJ & 142.30 & 40.30 & 21.43 & 2.57 & 87.94 & 32.09 & 33.85 & 2.33 \\ 
    &  & 250 & GPS & 107.05 & 28.84 & 19.33 & 1.96 & 59.57 & 22.68 & 33.32 & 2.63 \\ 
    &  & 250 & ICV & 106.91 & 27.38 & 19.33 & 1.86 & 59.49 & 21.82 & 33.39 & 2.43 \\ 
    &  & 250 & BDP & 155.31 & 49.36 & 21.19 & 2.88 & 100.03 & 39.10 & 33.95 & 2.47 \\ 
   3\% & 7.22\% & 500 & BCV & 231.72 & 52.48 & 27.43 & 2.58 & 88.00 & 23.35 & 29.27 & 1.61 \\ 
    &  & 500 & SJ & 254.06 & 55.76 & 28.41 & 2.72 & 97.89 & 24.41 & 29.85 & 1.72 \\ 
    &  & 500 & GPS & 209.60 & 43.00 & 26.43 & 2.20 & 78.13 & 19.35 & 28.69 & 1.36 \\ 
    &  & 500 & ICV & 224.87 & 49.94 & 27.13 & 2.48 & 84.83 & 22.18 & 29.07 & 1.52 \\ 
    &  & 500 & BDP & 295.24 & 72.06 & 29.34 & 3.10 & 115.42 & 30.61 & 30.47 & 2.01 \\ 
   7\% & 3.23\% & 500 & BCV & 142.85 & 32.31 & 21.54 & 2.02 & 88.49 & 26.36 & 34.24 & 2.23 \\ 
    &  & 500 & SJ & 156.60 & 34.33 & 22.31 & 2.13 & 99.64 & 27.51 & 33.92 & 2.13 \\ 
    &  & 500 & GPS & 129.24 & 26.48 & 20.76 & 1.73 & 77.04 & 21.69 & 34.58 & 2.21 \\ 
    &  & 500 & ICV & 138.64 & 30.75 & 21.30 & 1.95 & 84.79 & 25.11 & 34.38 & 2.25 \\ 
    &  & 500 & BDP & 181.98 & 44.37 & 23.04 & 2.43 & 120.50 & 35.80 & 33.86 & 2.09 \\ 

\noalign{\smallskip}\hline\noalign{\smallskip}
	\end{tabular}
	\end{table}

\section{Discussion}

A sampling plan methodology for a control-inspection policy is established that allows for independent as well as dependent sampling. Relying on a decision rule based on a $t$-type test statistic, sampling plans are constructed based on quantile estimates calculated from an additional sample taken from the production line. The new methodology applies to independent samples as well as dependent ones, under general
conditions.
When aggregating the available sampling information in order to minimize the required additional sampling costs at inspection time, it turns out that the relevant operating characteristics are relatively involved nonlinear equations that have to be solved numerically. Monte-Carlo simulations show that the approach works well and that the second stage sampling plan can be estimated with an accuracy that is comparable to the accuracy for the known formulas applicable for the first stage sampling plan. It also turns out that there is no uniformly superior method of bandwidth selection when relying on quantile estimates using inverted kernel density estimators. However, ICV as well as the GPS bandwidth selectors provide better results in many cases than more classical approaches. 

The extension of the acceptance sampling methodology to the case of $ L \ge 2 $ number of inspection time points, preferably allowing for dependent cluster sampling, requires further investigation. Firstly, the question arises whether or not one should design such procedures such that the overall type I and type II error rates are under control. Further, it remains an open issue to which extent one should aggregate data and to which extent time effects can be modelled stochastically. Lastly, for large $L$ appropriate procedures could resemble sequential (closed-end) procedures. 

Having in mind that in many cases present day quality control is based on highdimensional data arising from measurement curves and images such as IV curves or EL images in photovoltaics, the extension of the acceptance sampling methodology to  highdimensional and functional data deserves future research efforts as well; a deaper discussion is beyond the scope of the present article.


\section*{Acknowledgments}

The author thanks M.Sc. Andreas Sommer and M.Sc. Evgenii Sovetkin for proof-reading. Part of this work has been supported by a grant from the German Federal Ministry of the Environment, Nature Conservation and Nuclear Safety (grant no. 0325588B).

\section*{Appendix: Proofs}
\addcontentsline{toc}{section}{Appendix}

The results are obtained by refinements of the results obtained
in \citet{StelandZaehle2009} and \citet{MeisenPepelyshevSteland2012} and their extension to the two-stage setup with possible dependent samples. First, we need the two following auxiliary results, which are proved in \citet{MeisenPepelyshevSteland2012} for independent observations. However, it can be easily seen that the proofs work under more general conditions.

\begin{lemma} 
\label{t19:AuxLemma1}
If $ X_1, X_2, \dots $ have mean $\mu$, variance $ \sigma^2 \in (0, \infty) $ and satisfy a central limit theorem, i.e.
$
  \sqrt{n} \frac{ \overline{X}_n - \mu }{ \sigma} \stackrel{d}{\to} N(0,1),
$
as $n \to \infty $,  then
\[
  R_n = \sqrt{n} \frac{ \overline{X}_n - \mu }{ \sigma}  \frac{\sigma - S_m}{S_m}
  = o_P(1),
\]
as $ \min(n,m) \to \infty $, if $ S_m $ is a weakly consistent estimator for $ \sigma $.
\end{lemma}

\begin{lemma} 
\label{t19:AuxLemma2}
Suppose that
\begin{equation}
\label{t19:JointDConv}
  \left( \begin{array}{cc} \sqrt{m}( F_m^{-1}(p) - F^{-1}(p) ) \\
    \sqrt{n} \frac{\overline{X}_n - \mu }{ \sigma } 
    \end{array} \right)
   \stackrel{d}{\to} 
   \left( \begin{array}{cc} V_1 \\ V_2 \end{array} \right)
\end{equation}
as $ m \to \infty $, for a pair $ (V_1,V_2)' $ of random variables. Then
\[
  V_n = \sqrt{ \frac{n}{m} } \frac{ \sqrt{m}( F_m^{-1}(p) - F^{-1}(p) ) }{ S_m }
  = o_P(1),
\]
as $ \min(n,m) \to \infty $ such that $ n/m = o(1) $.
\end{lemma}

\begin{proof} (\textit{Theorem~\ref{t19:ApproxOCIndependent}})\\
In order to establish the approximations, first notice that the well known Skorohod/Dudley/Wichura representation theorem allows us to assume that all
distributional convergences can be assumed to hold a.s. and that all $ o_P(1) $ terms
are $ o(1) $; we leave the details to the reader. In particular, we may and shall assume that, almost surely,
\begin{equation}
\label{t19:ASConv}
  ( \overline{X}_1^*, \overline{X}_2^*)' \to (Z_1, Z_2)
  \quad \Leftrightarrow \quad
  \overline{X}_1^* - Z_1 = o(1), \ \overline{X}_2^* - Z_2 = o(1),
\end{equation}
 as $ \min(n_1, n_2) \to \infty $, where $ (Z_1, Z_2) $ are i.i.d standard normal random variables. Let us consider the probability
$
  q = P( T_{n1} >c_1, T_{n1} + T_{n2} > c_2 ).
$
As shown in detail below in the proof of Theorem~\ref{t19:ExpansionDependent1} for the more involved case of dependent sampling, we have the asymptotic expansions
\[
  \left( \begin{array}{cc} T_{n1} \\ T_{n2} \end{array} \right)
  =
    \left( 
      \begin{array}{cc} \overline{X}_1^* \\ \overline{X}_2^* \end{array}
    \right)
  -
  \left( \begin{array}{cc} \sqrt{n_1} G_m^{-1}(p) \\ \sqrt{n_2} G_m^{-1}(p) \end{array} \right)
  + o_P(1),
\]
as $ \min(n_1,n_2) \to \infty $ with $ n_1/n_2 \to \infty $ and $ \max(n_1,n_2)/m = o(1) $, and both coordinates are independent given $ G_m^{-1}(p) $.
Combing these expansions with (\ref{t19:ASConv}), we obtain, by plugging in the above expansions  and $ (Z_1, Z_2) $ for $ 
( \overline{X}_1^*, \overline{X}_2^* ) $,
\begin{align*}
  q & =
    P( \overline{X}_1^* - \sqrt{n_1} G_m^{-1}(p) + o(1) > c_1,
      \overline{X}_1^* + \overline{X}_2^* - (\sqrt{n_1}+\sqrt{n_2}) G_m^{-1}(p)
      + o(1) > c_2 ) \\
 & =   P( Z_1 - \sqrt{n_1} G_m^{-1}(p) + o(1) > c_1,
      Z_1 + Z_2 - (\sqrt{n_1}+\sqrt{n_2}) G_m^{-1}(p)
      + o(1) > c_2 ) 
\end{align*}
Conditioning on $ Z_2 = z_2 $ and $ X_{01}, \dots, X_{0m} $ leads to the expression
\[
     \int P( z > c_1 + \sqrt{n_1}G_m^{-1}(p) + o(1), 
Z_2 > c_2 - z + (\sqrt{n_1}+\sqrt{n_2}) G_m^{-1}(p) + o(1) ) \, d \Phi(z)
\]
for $ q $. Using $ E( 1_A 1_B ) = 1_A E( 1_B ) $, if $ A $ is non-random with respect to $P$, we obtain
\begin{align*}
   q  
& = \int_{c_1 + \sqrt{n_1}G_m^{-1}(p) + o(1)}^\infty 
[ 1 - \Phi( c_2 - z + (\sqrt{n_1}+\sqrt{n_2}) G_m^{-1}(p) + o(1) ) ] \, d \Phi(z) + o(1) \\
& = \int_{c_1 + \sqrt{n_1}G_m^{-1}(p) + o(1)}^\infty 
[ 1 - \Phi( c_2 - z + (\sqrt{n_1}+\sqrt{n_2}) G_m^{-1}(p)) ] \, d \Phi(z) + o(1),
\end{align*}
where we used the continuity of the integral. Further, the $ o(1) $ term in the integrand can be dropped by virtue of the Lipschitz continuity of $ \Phi $.
Combing the above results with the approximation
$  P(T_{n1} > c_1 ) = 1 - \Phi( c_1 + \sqrt{n_1} G_m^{-1}(p) ) + o(1),
$
establishes the result.
$\hfill \Box$
\end{proof}

We are now in a position to show Theorem~\ref{t19:ExpansionDependent1}. 
If $ X_{01}, \dots, X_{0m} $ and $ X_{11}, \dots, X_{1n_1} $ are independent,
then (\ref{t19:JointDConv}) follows easily. Otherwise, Assumption Q' ensures
the validity of the joint asymptotic normality for independent as well as
a large class of dependent sampling schemes.

\begin{proof} (\textit{Theorem~\ref{t19:ExpansionDependent1}})\\
Recall that $ E( \overline{X}_i ) = \mu $ and $ \operatorname{Var}( \overline{X}_i ) = \sigma/n_i $, $ i = 1, 2 $.
 We may closely follow the arguments given in 
 \citet{MeisenPepelyshevSteland2012}, since we have
  \[
    T_{ni} = \sqrt{n_i} \frac{ \overline{X}_i - \tau }{ S_m }
    = \sqrt{n_i} \frac{ \overline{X}_i - \mu }{ \sigma } + R_{ni} + \sqrt{n_i} \frac{\mu - \tau}{\sigma} + V_{ni}, 
  \]
  where 
  \begin{align*}
    R_{ni} 
      &= \sqrt{n_i} \frac{ \overline{X}_n - \mu }{ \sigma}  \frac{\sigma - S_m}{S_m}
      = o_P(1), \\
    V_{ni} & = \sqrt{n_i} \frac{\mu - \tau}{ \sigma } \left( \frac{\sigma}{S_m} - 1 \right) = o_P(1),
  \end{align*}
  as $ \min(n_i,m) \to \infty $, by virtue of Lemma~\ref{t19:AuxLemma1}, since
  $ \sqrt{m}( S_m - \sigma ) $ is asymptotically normal (by an application of the $ \Delta $-method, if the fourth moment is finite, and $ n_i / m = o(1) $, also see \citet{StelandZaehle2009}. Thus, it remains to consider
  \begin{align*}
    \sqrt{n_i} \frac{ \mu - F^{-1}(p) }{ \sigma }
    & = - \sqrt{n_i} G^{-1}(p) 
    = \sqrt{\frac{n_i}{m}} \sqrt{m}[G_m^{-1}(p) - G^{-1}(p) ]  - \sqrt{n_i} G_m^{-1}(p),
  \end{align*}
  where, by virtue of Assumption Q, the first term is $ o_P(1) $, if 
  $ \min(m,n_i) \to \infty $ and $ n_i / m = o(1) $. This shows the first
  assertion which is relevant when a quantile estimator of the standardized
  observations is available. Recall that $ \mu_0 = \mu + \Delta = E(X_0) $ and
  $ \sigma_0^2 = \operatorname{Var}(X_0) = \sigma^2 $.
  If a quantile estimator $ F_m^{-1} $ for the quantile function 
  $
    F_0^{-1}( p ) = \mu_0 + \sigma_0 G^{-1}(p)
  $
  of the additional sample taken at time $ t_0 $ is available, one proceeds as follows. 
  Noting that   $ \frac{ \mu - F^{-1}(p) }{ \sigma } = G^{-1}(p) = \frac{\mu_0 - F_0^{-1}(p)}{\sigma_0} $, we have
  \begin{align*}
  \sqrt{n_i} \frac{ \mu - F^{-1}(p) }{ \sigma }
    & = \sqrt{ \frac{n_i}{m} } \frac{ \sqrt{m}[F_m^{-1}(p) - F_0^{-1}(p) ] }{ \sigma_0 }
    - \sqrt{n_i} \frac{ F_m^{-1}(p) - \overline{X}_0 }{ S_m } \\
    & \qquad
    - \sqrt{n_i} \frac{ F_m^{-1}(p) - \overline{X}_0 }{ S_m }
    \left( \frac{S_m}{\sigma_0} - 1 \right) 
    + \sqrt{m} \frac{ \mu_0 - \overline{X}_0 }{ \sigma_0 } \sqrt{ \frac{n_i}{m} }.
  \end{align*}
  In this decomposition at the right side the first, third and fourth term 
  are $ o_P(1) $, as $ \min(n_i,m) \to \infty $ with $n_i/m = o(1) $, $ i = 1, 2 $. 
  Notice that the fourth term is $ o(1) $, since 
  \[ 
    \sqrt{m}(\overline{X}_0 - \mu_0)/\sigma_0 \stackrel{d}{\to} N(0,1),
  \]
  if $ X_{01}, \dots, X_{0m} $ are i.i.d. $ \sim F((\bullet-\mu_0)/\sigma_0) $ or as a consequence of Assumption Q'. 
  Thus,
  \[
    \sqrt{n_i} \frac{ \mu - F^{-1}(p) }{ \sigma }
    = \sqrt{n_i} G_m^{-1}(p) + o_P(1),
  \]
  where now $ G_m^{-1}(p) = \frac{ F_m^{-1}(p) - \overline{X}_0 }{ S_m } $
  is an estimator of the quantile function $ G^{-1}(p) $ of the standardized observations,
  see Remark~\ref{t19:RemarkQuantEst}. $ \hfill \Box $
\end{proof}

\begin{proof} (\textit{Proposition~\ref{t19:CLTDependent1}})\\
We consider the case $n_1<n_2$. W.l.o.g. we can assume that $X_{21},\ldots,X_{2n_1}$ are the time $t_2$ measurements from those $ n_1$ items (modules) already drawn at time $t_1$, and $X_{2,n_1+1},\ldots,X_{2n_2}$ are $n_2-n_1$ measurements taken from newly selected items  from the lot. By virtue of the Cram\'{e}r-Wold device, to prove the proposition, it suffices to show that for all constants $d_1,d_2\in\mathbb{R} $ with $(d_1,d_2)\neq(0,0)$
\[
d_1\overline{X}_1^*+d_2\overline{X}_2^*\overset{d}{\underset{n\rightarrow\infty}{\rightarrow}}N(0,d_1^2+d_2^2+2d_1d_2\sqrt{\lambda}\varrho'),
\]
since
$E(d_1\overline{X}_1^*+d_2\overline{X}_2^*)=0 \text{ and}
$
\[Var (d_1\overline{X}_1^*+d_2\overline{X}_2^*)=d_1^2+d_2^2+2d_1d_2\sqrt{\frac{n_1}{n_2}}\varrho'.
\]
Write
$d_1\overline{X}_1^*+d_2\overline{X}_2^*=A_n+B_n, $
where
\begin{align*}
&A_n=\frac{1}{\sqrt{n}}\sum_{j=1}^{n_1}\left[d_1\frac{X_{1j}-\mu_1}{\sigma_1}+d_2\sqrt{\frac{n_1}{n_2}}\frac{X_{2j}-\mu_1}{\sigma_1}\right],\\
&B_n= \frac{d_2}{\sqrt{n_2}}\sqrt{n_2-n_1}\frac{1}{\sqrt{n_2-n_1}}\sum_{j=n_1+1}^{n_2}\frac{X_{2j}-\mu_2}{\sigma_2}.
\end{align*}
The summands of $A_n$ form an array of row-wise independent random variables $\xi_{n_1,j}, 1\leq j\leq n_1,n_1\geq 1$, with mean zero and variance
\[Var(\xi_{n1,j})= d_1^2+d_2^2\frac{n_1}{n_2}+2d_1d_2\sqrt{\frac{n_1}{n_2}}\varrho'\rightarrow d_1^2+d_2^2\lambda+2d_1d_2\sqrt{\lambda}\varrho',
\]
as $ n_1\rightarrow\infty.$
Further, it is easy to verify that
$B_n\stackrel{d}{\rightarrow}N(0,d_2^2(1-\lambda)), $
as $n_1\rightarrow\infty$. By independence of $A_n$ and $B_n$, we obtain
\[
\left(\begin{array}{c}A_n\\B_n\end{array}\right)\stackrel{d}{\rightarrow}N\left(\left(\begin{array}{c}0\\0\end{array}\right),\left(\begin{array}{cc}d_1^2+d_2^2\lambda+2d_1d_2\sqrt{\lambda}\varrho' & 0\\0& d_2^2(1-\lambda)\end{array}\right)\right),
\]
as $n_1\rightarrow\infty.$ Now the continuous mapping theorem entails
\[A_n+B_n \stackrel{d}{\rightarrow} N(0,\sigma_{AB}^2),
\]
as $n_1\rightarrow\infty$, where
$
\sigma_{AB}^2= d_1^2+d_2^2\lambda+2d_1d_2\sqrt{\lambda}\varrho'+d_2^2(1-\lambda)
=d_1^2+d_2^2\lambda+2d_1d_2\sqrt{\lambda}\varrho',
$
which establishes the assertion.
$\hfill\Box$
\end{proof}

\begin{proof} (\textit{Theorem~\ref{t19:ApproximationDependent1}})\\
The proof goes along the lines of the proof for the independent case.
Again we may and shall assume that the distributional convergence is
a.s. and $ o_P(1) $ are $ o(1) $ a.s. Therefore,
$
  \left( \overline{X}_1^*, \overline{X}_2^*\right)
  \stackrel{a.s.}{\to} \left(  Z_1,  Z_2 \right),
$
as $ \min(n_1,n_2) \to \infty $.  Here $ (Z_1,Z_2) $ is a bivariate random vector
that is jointly normal with mean $0$, unit variances and correlation $ \rho $.
The probability
$  q = P( T_{n1} >c_1, T_{n1} + T_{n2} > c_2 )
$ can now be calculated as follows. We have
\begin{align*}
  q & =
    P( \overline{X}_1^* - \sqrt{n_1} G_m^{-1}(p) + o(1) > c_1,
      \overline{X}_1^* + \overline{X}_2^* - (\sqrt{n_1}+\sqrt{n_2}) G_m^{-1}(p)
      + o(1) > c_2 ) \\
     & = 
     P( Z_1 > c_1 + \sqrt{n_1}G_m^{-1}(p) + o(1), 
    Z_2 > c_2 - z + (\sqrt{n_1}+\sqrt{n_2}) G_m^{-1}(p) + o(1) )  \\
  & = \int 1( z >c_1 + \sqrt{n_1} G_m^{-1}(p) + o(1)) \\ 
  & \qquad \qquad
  P( Z_2 > c_2 - z + (\sqrt{n_1}+\sqrt{n_2}) G_m^{-1}(p) + o(1) | Z_1=z) \, d \Phi(z) + o(1).
\end{align*}
However, now we have to take into account that the conditional law of $ Z_2 $ given $ Z_1=z $ is a normal distribution that depends on $ z $, namely with mean $ \rho z $ and variance $ 1-\rho^2 $. Therefore, we may conclude that, up to an $ o(1) $ term,
\begin{align*}
  q &= \frac{1}{\sqrt{2\pi}}
    \int_{c_1 + \sqrt{n_1} G_m^{-1}(p)}^\infty
    \left[ 1 - \Phi\left( \frac{ c - z + (\sqrt{n_1}+\sqrt{n_2}) G_m^{-1}(p) - \rho z }{ \sqrt{1-\rho^2} } \right) \right] e^{-z^2/2} \, dz.
\end{align*}
The unknown correlation parameter $ \rho $ may be replaced by its consistent
estimator $ \widehat{\rho} $, since the integrand is Lipschitz continuous, if
$ |\rho| \le \overline{\rho} < 1 $. Indeed, observing that
\begin{align*}
 & \frac{d}{d \rho} \Phi\left( \frac{ c - z + (\sqrt{n_1}+\sqrt{n_2}) G_m^{-1}(p) - \rho z }{ \sqrt{1-\rho^2} } \right)  \\
 & = \varphi_{(0,1)}\left( \frac{ c - z + (\sqrt{n_1}+\sqrt{n_2}) G_m^{-1}(p) - \rho z }{ \sqrt{1-\rho^2} } \right) \\
 & \cdot \frac{-z}{\sqrt{1-\rho^2}} + \rho
   \frac{ c - z + (\sqrt{n_1}+\sqrt{n_2}) G_m^{-1}(p) - \rho^2 }{ (\sqrt{1-\rho^2})^3 },
\end{align*}
where $ \varphi_{(0,1)} $ denotes the density of the $ N(0,1) $-distribution,
we can find $ 0 < c < \infty $, such that the above expression is not larger than
$c |z|  $, as a function of $z$. Hence, replacing $ \rho $ by its estimator $ \widehat{\rho}_n $ results in an error term that can be bounded by $ (2\pi)^{-1} c \int |z| e^{-z^2/2} \, dz | \widehat{\rho}_n - \rho | = o_P(1)$. 
Putting things together, we arrive at the assertion of the theorem.
$\hfill \Box$
\end{proof}

\end{document}